\documentclass{article}

\usepackage{amssymb,amsthm,amsmath}
\usepackage{stmaryrd}
\usepackage{url}

\newcommand{\simf}{s.i.m.f.}
\newcommand{\spimf}{s.p.i.m.f.}
\newcommand{\C}{\mathbb{C}}
\newcommand{\Q}{\mathbb{Q}}
\newcommand{\Z}{\mathbb{Z}}
\newcommand{\N}{\mathbb{N}}
\newcommand{\F}{\mathbb{F}}
\newcommand{\PS}{\mathbb{P}}
\newcommand{\R}{\mathbb{R}}
\newcommand{\B}{\mathcal{B}^o}
\newcommand{\Forms}{\mathcal{F}}
\newcommand{\calH}{\mathcal{H}}
\newcommand{\ZZ}{\mathcal{Z}}
\newcommand{\GL}{\textnormal{GL}}
\newcommand{\Sp}{\textnormal{Sp}}
\newcommand{\abs}[1]{\lvert #1\rvert}
\newcommand{\tensor}{\otimes}
\newcommand{\isom}{\simeq}
\newcommand{\tpForms}{\Forms_{>0}}
\newcommand{\DP}{\! : \!}
\newcommand{\MAGMA}{{\sc MAGMA}}
\newcommand{\transpose}[1]{#1^{\textnormal{tr}}}
\newcommand{\mtranspose}[1]{#1^{-\textnormal{tr}}}
\newcommand{\tensordown}[1]{\underset{\text{\scriptsize $ #1 $}}{\tensor}}

\DeclareMathOperator{\Diag}{Diag}
\DeclareMathOperator{\Nr}{Nr}
\DeclareMathOperator{\nr}{nr}
\DeclareMathOperator{\Tr}{Tr}
\DeclareMathOperator{\Out}{Out}
\DeclareMathOperator{\End}{End}
\DeclareMathOperator{\Aut}{Aut}
\DeclareMathOperator{\Stab}{Stab}
\DeclareMathOperator{\Syl}{Syl}
\DeclareMathOperator{\Alt}{Alt}
\DeclareMathOperator{\SL}{SL}
\DeclareMathOperator{\PGL}{PGL}
\DeclareMathOperator{\LGruppe}{L}
\DeclareMathOperator{\Gal}{Gal}

\theoremstyle{plain}
\newtheorem{theorem}{Theorem}
\newtheorem{lemma}[theorem]{Lemma}
\newtheorem{prop}[theorem]{Proposition}
\newtheorem{definition}[theorem]{Definition}
\newtheorem{corollary}[theorem]{Corollary}
\theoremstyle{remark}
\newtheorem{remark}[theorem]{Remark}
\newtheorem{example}[theorem]{Example}
\numberwithin{theorem}{section}

\begin{document}

\title{Finite symplectic matrix groups}

\author{Markus Kirschmer\thanks{Lehrstuhl D f\"ur Mathematik, RWTH Aachen University, Templergraben 64, 52062 Aachen, Germany, markus.kirschmer@math.rwth-aachen.de}}
%\address{Lehrstuhl D f\"ur Mathematik, RWTH Aachen University, Templergraben 64, 52062 Aachen, Germany}
%\email{markus.kirschmer@math.rwth-aachen.de}
\date{28 August 2009}
%\keywords{}

\maketitle 

\begin{abstract}
This paper classifies the maximal finite subgroups of $\Sp_{2n}(\Q)$ for $1 \le n \le 11$ up to conjugacy in $\GL_{2n}(\Q)$.
\end{abstract}

\section{Introduction}

The maximal finite subgroups of $\Sp_{2n}(\Q)$ are maximal finite subgroups of $\GL_m(K)$ for minimal totally complex subfields $K$ with $2n=m \cdot [K:\Q]$. Moreover, the maximal finite subgroups of $\Sp_{2n}(\Q)$ are full automorphism groups of Euclidean lattices fixing an additional nondegenerate symplectic form. So the classification yields highly symmetric 
symplectic structures on interesting Euclidean lattices.

The (conjugacy classes of) maximal finite subgroups of $\GL_{m}(\Q)$ have been classified up to $m=31$ in a series of papers \cite{Cryst4,Plesken_app_repr,FRMG,ERMG,FRMG25}.
% In this paper, we classify the maximal finite subgroups of $\Sp_{2n}(\Q)$ up to conjugacy in $\GL_{2n}(\Q)$ for $1 \le n \le 11$ in my thesis \cite{MK}.

The strategy for these classifications is as follows. First, it suffices to classify only the maximal finite (symplectic) matrix groups whose natural representation is irreducible over $\Q$. Then one reduces the problem to the so-called (symplectic) primitive matrix groups. The concept of primitivity is the key ingredient for these classifications, since it has important consequences for normal subgroups. Most notably, the restriction of the natural representation of an irreducible (symplectic) primitive matrix group $G<\GL_{m}(\Q)$ onto a normal subgroup $N$ splits into copies of a single irreducible representation of $N$. In particular, each abelian subgroup of $O_p(G)$ must be cyclic. Using a theorem of Philip Hall, this restricts the Fitting subgroup $F(G)$ of $G$ to a finite number of candidates depending only on $m$.
Using the Atlas of Finite Simple Groups or the classification of Hiss and Malle \cite{HissMalle} one gets the candidates for the layer $E(G)$ of $G$. Hence there are only finitely many candidates for the generalized Fitting subgroup $F^*(G)$. Then $G$ can be constructed from $F^*(G)$ using group theory, cohomology and the fact that $G/F^*(G) \le \Out(F^*(G))$.

However, these constructions are quite cumbersome and error-prone. Moreover, they do not use the fact that $G$ is a maximal finite (symplectic) subgroup of $\GL_{m}(\Q)$. In many cases, this method can be avoided using algebraic number theory and the fact that $G$ is the full automorphism group of every $G$-invariant lattice.

The main result of this classification is the following
\paragraph*{Theorem}
The number of maximal finite subgroups of $\Sp_{2n}(\Q)$ (up to conjugacy in $\GL_{2n}(\Q)$) are given by
$$
\begin{array}{c|ccccccccccc}
2n & 2 & 4 & 6 & 8 & 10 & 12 & 14 & 16 & 18 & 20 & 22 \\ \hline
\# \mbox{ primitive classes} & 2& 5& 2& 21& 3& 23& 3& 63& 6& 26& 4  \\
\# \mbox{ classes (total)} & 2 & 6 & 4 & 28 & 5 & 32 & 5 & 91 & 10 & 36 & 6\\
\end{array}
$$
Representatives of these conjugacy classes are available online in a \MAGMA\ \cite{Magma} readable format from\\
\url{http://www.math.rwth-aachen.de/~Markus.Kirschmer/symplectic/} . Section \ref{sectab} explains how to access the groups in the database and how to find a representative in the database that is conjugate to a given symplectic irreducible maximal finite (\simf) matrix group.
A proof of completeness is given in my thesis \cite{MK}. The methods used are explained here as they might be of independent interest and can be applied to classify symplectic matrix groups also in higher degree.
The existence of the famous Leech lattice makes dimension 24  very interesting, 
though a full classification of \simf\ subgroups of $\Sp_{24}(\Q) $ 
 would be quite tedious due to the large number of possible Fitting subgroups.

This paper is organized as follows. Section \ref{secprop} contains some basic definitions for matrix groups. It is shown that it suffices to classify only the maximal finite symplectic matrix groups whose natural representation is irreducible over the rationals. These symplectic irreducible maximal finite (\simf) matrix groups can be characterized by automorphism groups of lattices.

Section \ref{secprim} shows that it suffices to classify only the so-called symplectic primitive matrix groups and it contains the general outline of the classification. The definition and some properties of the generalized Bravais groups are also recalled.

Section \ref{secfam} introduces some infinite families of maximal finite symplectic matrix groups. In particular, all maximal finite symplectic subgroups of $\GL_{p-1}(\Q)$ and $\GL_{p+1}(\Q)$ whose orders are divisible by some prime $p$ are determined. 

Section \ref{sec2p} shows that if $p\ge 5$ is prime, then the Fitting subgroup of a \spimf\ subgroup of $\GL_{2p}(\Q)$ is cyclic of order $2,4$ or $6$. In particular, these groups can be constructed by the methods of Section \ref{secmet}.

Section \ref{secmet} describes some shortcuts that are used frequently in the classification. It gives two results that can be used to rule out some candidates for normal subgroups of index $2^k$. In this section, the $m$-parameter argument is recalled which can be used to find all \simf\ matrix groups that contain a given matrix group whose commuting algebra is a field. Finally, it is explained how to find all \simf\ matrix groups that contain a given normal subgroup whose commuting algebra is a quaternion algebra.

\section{Some definitions and basic properties}\label{secprop}

\begin{definition}
Let $G < \GL_m(\Q)$. 
\begin{enumerate}
\item $\Forms(G) = \{ F \in \Q^{m \times m} \mid gF\transpose{g} = F \mbox{ for all } g \in G\}$ is called the \emph{form space} of $G$. Further $\Forms_{sym},\; \Forms_{>0}$ and $\Forms_{skew}$ denote the subsets of symmetric, positive definite symmetric and skewsymmetric forms.
\item The group $G$ is called \emph{symplectic} if $\Forms_{skew}(G)$ contains an invertible element.
\item $\ZZ(G) = \{ L \subset \Q^{1\times m} \mid L \mbox{ a rank $m$ lattice with $Lg=L$ for all $g\in G$} \}$ is the set of (full) \emph{$G$-invariant lattices}.
\item $\End(G)= \{e \in \Q^{m \times m}\mid eg=ge \mbox{ for all } g \in G\}$ is the commuting algebra (or endomorphism ring) of $G$.
\item If $K$ is a subfield of $\End(G)$ then 
$$\Aut_K(L,F) = \{g \in \Q^{m \times m} \mid Lg=L,\; g F \transpose{g} = F \mbox{ and } ge=eg \mbox{ for all } e \in K\}$$
denotes the group of $K$-linear automorphism of $(L, F)$. If $K \isom \Q$ the subscript $K$ is usually omitted.
\end{enumerate}
\end{definition}

An averaging argument shows that $G < \GL_m(\Q)$ is finite if and only if $\tpForms(G)$ and $\ZZ(G)$ are both nonempty. Thus, the finite subgroup $G < \GL_m(\Q)$ is maximal finite if and only if $G = \Aut(L, F)$ for all $(L, F) \in \ZZ(G) \times \tpForms(G)$.

If $F \in \Forms(G)$ is invertible, then $\End(G) \to \Forms(G),\; e \mapsto eF$ is an isomorphism of $\Q$-spaces. This can be used to characterize the  maximal finite symplectic matrix groups as automorphism groups of lattices.

Before this result is stated, the next lemma shows that it suffices to consider only rationally irreducible matrix groups, i.e. matrix groups whose natural representation is irreducible over the rationals.

\begin{lemma}\label{irred_OK}
The natural representation $\triangle\colon G \to \GL_{2n}(\Q),\; g \mapsto g$ of a maximal finite symplectic subgroup $G < \GL_{2n}(\Q)$ is the direct sum of rationally irreducible and pairwise nonisomorphic representations of $G$.
\end{lemma}
\begin{proof}
Suppose $\triangle$ decomposes into $\oplus_{i=1}^s n_i \triangle_i$ where $\triangle_1, \dots,\triangle_s$ are rationally irreducible and pairwise nonisomorphic.
Since the endomorphism ring and thus the form space of $G$ decomposes into direct sums, it follows that each group $n_i \triangle_i(G)$ is maximal finite symplectic. Hence it suffices to show that $n_i = 1$ for all $i$. Let $E_i = \End(\triangle_i(G))$ be the commuting algebra of $\triangle_i(G)$. If $E_i$ is not a totally real number field then $\triangle_i(G)$ is already symplectic (see Lemma \ref{char_symp}) and thus $n_i \triangle_i(G)$ is contained in the wreath product $ \triangle_i(G) \wr S_{n_i}$. If $E_i$ is a totally real number field then $n_i$ must be even since the form space of $n_i \triangle_i(G)$ contains an invertible skewsymmetric element. But then $n_i \triangle_i(G)$ is conjugate to a subgroup of the tensor product $\triangle_i(G) \tensor H$ for any maximal finite symplectic $H < \GL_{n_i}(\Q)$.
\end{proof}

\begin{lemma}\label{char_symp}
Let $G < \GL_{2n}(\Q)$ be finite and rationally irreducible. The following statements are equivalent:
\begin{enumerate}
\item $G$ is symplectic.
\item $\End(G)$ contains some (minimal) totally complex subfield.
\item There exists some (minimal) totally complex field $K$ of degree $d$ and some representation $\triangle_K \colon G \to \GL_{\frac{2n}{d}}(K)$ such that the natural character of $G$ is the trace character corresponding to $\triangle_K$.
\end{enumerate}
\end{lemma}
\begin{proof}
Let $F \in \tpForms(G)$. The space $\Forms(G)$ is closed under taking transposes, hence it decomposes into $\Forms_{sym}(G) \oplus \Forms_{skew}(G)$. Since
$E:= \End(G)$ is a skewfield it follows from $\Forms(G)=E\cdot F$ that $G$ is symplectic if and only if $\Forms(G) \ne \Forms_{sym}(G)$ which is equivalent to say that $\Forms_{skew}(G)\ne \{0\}$.

If $e\in E$ such that $eF$ is symmetric (skewsymmetric) then $e$ is a selfadjoint (antiselfadjoint) endomorphism of the Euclidean space $(\R^n, F)$. Hence $\Q[e]\le E$ is totally real (complex). Thus the first two statements are equivalent. The equivalence of the last two statements follows from ordinary representation theory.
\end{proof}

The symplectic irreducible maximal finite (\simf) matrix groups can now be characterized as follows.
\begin{corollary}
A finite rationally irreducible symplectic matrix group $G < \GL_{2n}(\Q)$ is \simf\ if and only if $G=\Aut_K(L, F) $ for all $(L, F) \in \ZZ(G) \times \tpForms(G)$ and all minimal totally complex subfields of $\End(G)$.
\end{corollary}

It follows for example from \cite[Theorems 3.7 and 3.25]{Artin} that $\Q^{m \times m}$ contains an invertible skewsymmetric matrix $A$ if and only if $m$ is even. Moreover, if $m=2n$, then there exists some $T \in \GL_{2n}(\Q)$  such that $T A \transpose{T}=J_n$ where $J_n:= \left( \mbox{\small $\begin{array}{rr} 0 & I_n\\ -I_n & 0  \end{array}$} \right)$. Thus each $\GL_{2n}(\Q)$-conjugacy class of symplectic subgroups of $\GL_{2n}(\Q)$ contains a representative in 
$$\Sp_{2n}(\Q) := \{g \in \GL_{2n}(\Q) \mid g J_n \transpose{g} = J_n\}\:.$$ 

If $G < \Sp_{2n}(\Q)$, one might ask how $\{G^x \mid x \in \GL_{2n}(\Q) \mbox{ and } G^x < \Sp_{2n}(\Q)\}$ decomposes into $\Sp_{2n}(\Q)$-conjugacy classes. The following example shows that there are infinitely many such classes.

\begin{example}
Let $g = \mbox{\small $\left(\begin{array}{@{}r@{\;\;}r@{}} 0 & 1 \\ -1 & 0 \end{array}\right)$}$ and $G = \left\langle g \right\rangle < \Sp_{2}(\Q)$. For $a \in \Q$ set $h_a = \Diag(1, a)$. Then $G^{h_a} < \Sp_2(\Q)$ for all such $a$. Moreover, $G^{h_a}$ is conjugate to $G^{h_b}$ in $\Sp_2(\Q)$ if and only if $\abs{ab} \in \Nr_{\Q(i)/\Q}(\Q(i)^*)$.
\end{example}
\begin{proof}
Let $t:= \mbox{\small $\left(\begin{array}{@{}r@{\;\;}r@{}} 0 & 1 \\ 1 & 0 \end{array}\right)$}$. Then $G^{h_a}$ is conjugate to $G^{h_b}$ in $\Sp_2(\Q)$ if and only if there exists some $x \in \Sp_{2n}(\Q)$ such that $g^{h_a} = \pm g^{h_b x} = g^{t^k h_b x}$ with $k \in \{0,1\}$. This is equivalent to say that $\alpha:= t^k h_b x h_{a^{-1}}$ is contained in the commuting algebra of $G$, which is isomorphic to $\Q(i)$. Taking determinants yields $(-1)^k \cdot b/a = \Nr_{\Q(i)/\Q}(\alpha) > 0$ as claimed.
\end{proof}

\section{Primitivity}\label{secprim}

Lemma \ref{irred_OK} shows that it suffices to classify only the (conjugacy classes of) rationally irreducible maximal finite symplectic matrix groups.
This set can be restricted further as follows.

\begin{definition}
Let $K$ be a field. An irreducible matrix group $G < \GL_m(K)$ is called primitive if it is not conjugate to some subgroup $H \wr S_k$ for some $H < \GL_d(K)$ with $kd=m$.

Similarly, an irreducible symplectic subgroup $G < \GL_{2n}(\Q)$ is called symplectic primitive, if it is not conjugate to a subgroup of $H \wr S_k$ for some $H < \Sp_{2d}(\Q)$ with $kd= n$.
\end{definition}

Clearly, the conjugacy classes of symplectic imprimitive matrix groups can be constructed from the classifications of smaller dimensions. Furthermore, an irreducible symplectic matrix group $H$ is imprimitive if and only if there exists some invariant lattice $L \in \ZZ(H)$ that admits a nontrivial decomposition $L = \perp_i L_i$ where the components $L_i$ are pairwise perpendicular with respect to the full form space $\Forms(H)$. Thus, symplectic imprimitive groups can be recognized easily.

The concept of primitivity has some important consequences for normal subgroups.

\begin{remark}
Let $G < \GL_m(K)$ be primitive and $N \unlhd G$. Then $G$ acts on $N$ by conjugation. Hence it also acts on the set of central primitive idempotents of $\left \langle N\right\rangle_K$. Thus $G$ permutes the homogeneous components of the natural $KN$-module $K^{1\times m}$. But since $G$ is primitive, there can only be one such component. Hence the natural $KN$-module $K^{1\times m}$ splits into a direct sum of $k$ isomorphic $KN$-modules of dimension $\frac{m}{k}$.
\end{remark}

\begin{lemma}\label{mainlemma}
Let $G < \GL_{2n}(\Q)$ be irreducible and symplectic primitive.
\begin{enumerate}
\item If $N \unlhd G$ then the natural character of $N$ is a multiple of a single rationally irreducible character. 
\item If $O_p(G) \ne 1$ then $p-1$ divides $2n$.
\item Every abelian characteristic subgroup of $O_p(G)$ is cyclic.
\end{enumerate}
\end{lemma}
\begin{proof}
The last two statements are immediate consequences of the first one. 

Let $K< \End(G)$ be a minimal totally complex subfield of degree $d$. Let $\{f_1, \dots , f_r\}$ and $\{e_1,\dots,e_s\}$ be the central primitive idempotents of the enveloping algebras $\left\langle N \right\rangle_\Q$ and $\left\langle G \right \rangle_{K}$ respectively. Let $\chi_i$ denote the character corresponding to a simple $\left\langle G \right \rangle_{K}e_i$ module and let $L= \Q(\chi_1,\dots,\chi_s) \subseteq K$ be their character field. Then $L/\Q$ is Galois and $e_i = \frac{\chi_i(1)}{\abs{G}} \sum_{g \in G} \chi_i(g^{-1}) g \in \left\langle G \right \rangle_{L}$. Since $G$ is irreducible, $\{e_1,\dots, e_s\}$ is a Galois orbit under $\Gal(L/\Q)$. For any $1 \le j \le r$ there exists some $i$ such that $e_if_j \ne 0$. Since $f_j \in \left\langle N\right\rangle_\Q$ is fixed under $\Gal(L/\Q)$, it follows that $e_if_j \ne 0$ for all $i,j$.
Denote by $\triangle_K\colon G \to \GL_{\frac{2n}{d}}(K)$ an irreducible representation of $G$ as in Lemma \ref{char_symp}. The enveloping algebra $\left\langle \triangle_K(G) \right \rangle_{K}$ is isomorphic to $\left\langle G \right \rangle_{K} e_i$ for some $i$. Now  $\{e_if_1, \dots, e_if_r\}$ is a set of central idempotents of $\left\langle N \right \rangle_{K} e_i \isom \left\langle \triangle_K(N) \right \rangle_{K}$. But since $G$ is symplectic primitive, $\triangle_K(G)$ is primitive as well. Therefore $\left\langle\triangle_K(N) \right\rangle_K$ is a simple algebra by the remark above. This shows $r=1$, since no $e_if_j$ vanishes.
\end{proof}

In particular, the candidates for $O_p(G)$ are well known by a theorem of Ph. Hall:

\begin{theorem}[Hall, \protect{\cite[Satz 13.10]{Huppert1}}]\label{Hall}
Suppose $P$ is a $p$-group such that every abelian characteristic subgroup of $P$ is cyclic. Then $P$ is the central product of two subgroups $E_1$ and $E_2$ where either
\begin{enumerate}
\item $p\ne 2$ and $E_1$ is extraspecial of exponent $p$ and $E_2$ is cyclic.
\item $E_1 \isom 2^{1+2m}_+$ and $E_2$ is cyclic, dihedral, quasi-dihedral or a generalized quaternion $2$-group. 
\item $E_1 = 2^{1+2m}_-$ and $E_2 = 1$.
\end{enumerate} 
\end{theorem}

The outline of the classification of all conjugacy classes of \simf\ subgroups of $\Sp_{2n}(\Q)$ is now as follows:
\begin{enumerate}
\item The symplectic imprimitive matrix groups come from the classifications of $\Sp_{2d}(\Q)$ where $d$ runs through all divisors of $n$. The wreath products $H \wr S_k$ for some \spimf\ $H < \Sp_{2d}(\Q)$ with $dk=n$ are most often maximal finite. (The only exception up to dimension $2n=22$ is $C_4 \wr S_2 < \Sp_{4}(\Q)$).

Suppose now $G < \Sp_{2n}(\Q)$ is \spimf{}.
\item There are only finitely many candidates for the Fitting subgroup $F(G)$ according to Hall's classification.
\item There are only finitely many candidates for the layer $E(G)$ (the central product of all subnormal quasisimple subgroups of $G$). These are described in \cite{HissMalle} which is based on the ATLAS \cite{atlas}. (Note that this step depends on the completeness of the classification of all finite simple groups. For small dimensions, one can also use the results of Blichfeldt, Brauer, Lindsey, Wales and Feit (see for example \cite{Feit2}) who classified the finite (quasiprimitive) subgroups of $\GL_m(\C)$ for $m \le 10$).
\item So there are only finitely many candidates for the generalized Fitting subgroup $F^*(G)$, i.e. the central product of $E(G)$ and $F(G)$. Since $G/F^*(G)$ is isomorphic to a subgroup of $\Out(F^*(G))$ it remains to construct all possible extensions of $F^*(G)$
 up to conjugacy in $\GL_{2n}(\Q)$.
\end{enumerate}
The last step is the crucial one. Although it is a cohomological task to find all abstract extensions $G$, this does not yield all matrix group extension of $F^*(G)$. A first step is to replace $F^*(G)$ by its generalized Bravais group as described below.

\subsection{Generalized Bravais groups}

Let $G< \Sp_{2n}(\Q)$ or $G < \GL_m(\Q)$ be maximal finite. If $N \unlhd G$ then clearly, $N \unlhd \B(N) :=\{g \in G \mid g \text{ centralizes } \End(N)\} \unlhd G$. 
The so-called generalized Bravais group $\B(N)$ can be constructed  independently of the group $G$ as described below. This was first introduced in \cite[page 82]{FRMG}.
This gives a general construction how to replace candidates for normal subgroups of $G$ by (hopefully) larger normal subgroups.

First, one applies the so called radical idealizer process to the $\Z$-order $\Lambda_0 := \left \langle N \right \rangle_\Z$ in the enveloping algebra $ \left \langle N \right \rangle_\Q$ of $N$. I.e. one iterates the following steps
\begin{itemize}
\item Let $I_i$ be the arithmetical radical of $\Lambda_i$, that is the intersection of all maximal right ideals of $\Lambda_i$ that contain the (reduced) discriminant of $\Lambda_i$.
\item Let $\Lambda_{i+1}$ be the right idealizer of $I_i$ in $ \left \langle N \right \rangle_\Q$.
\end{itemize}
This process stabilizes in a necessarily hereditary order $\Lambda_\infty$ see \cite[Theorems 39.11, 39.14 and 40.5]{Reiner}.

The natural $\Q N$-module $\Q^{1\times 2n}$ is isomorphic to a multiple of an irreducible $\Q N$-module $V$ according to Lemma \ref{mainlemma}.  Let $L_1,\dots,L_s$ represent the isomorphism classes of $\Lambda_\infty$-lattices in $V$ and fix some $F \in \tpForms(N)$. Then 
$$\B(N) = \{ g \in \left \langle N \right \rangle_\Q \mid L_ig = L_i \mbox{ for all $1\le i \le s$ and } g F \transpose{g} = F \}\:.$$

The group $\B(N)$ is finite, since it fixes a full lattice as well as a positive definite form.
By the double centralizer property $\B(N)$ and $N$ have the same commuting algebras and thus the same form spaces. In particular, the definition of $\B(N)$ does not depend on the choice of the form $F$. 
Moreover, $\B(N)$ has the following properties:

\begin{lemma}[\protect{\cite[Proposition II.10]{FRMG}}]\label{lemmabravais}
If $G$ is a \spimf\ subgroup of $\Sp_{2n}(\Q)$ and $N \unlhd G$ then
\begin{enumerate}
\item $N \unlhd \B(N) \unlhd G$
\item If $X$ is a finite subgroup of $\left \langle N \right \rangle_\Q^*$ such that $N \unlhd X$ then $X \le \B(N)$.
\end{enumerate} 
\end{lemma}
By the double centralizer property, it follows that if $G$ is \spimf\ then $\B(N) = \{g \in G \mid\text{ centralizes } \End(N) \}$ as claimed in the beginning of this section.

The candidates of $O_p(G)$ for some irreducible symplectic primitive $G< \GL_{2n}(\Q)$ are well known by Theorem \ref{Hall}. Their generalized Bravais groups have been determined generically as follows.
\begin{prop}\cite[Chapter 8]{FQMG}
The candidates for $N = O_p(G)$ and their generalized Bravais groups for an irreducible symplectic primitive matrix groups $G$ are given by 
$$
\begin{array}{|l|c|c|c|} \hline
N                                       & \B(N)                 & \dim_\Q(\left \langle N\right\rangle_\Q) & \End(N)    \\ \hline
C_{p^m}                                 & \pm N                 & p^{m-1}(p-1)          & \Q(\zeta_{p^m})       \\ \hline
p_+^{1+2n},\quad (p>2)                  & \pm N.\Sp_{2n}(p)     & p^n(p-1)              & \Q(\zeta_p)           \\ \hline
2_+^{1+2n}                              & N.O^+_{2n}(2)         & 2^n                   & \Q                    \\ \hline
2_-^{1+2n}                              & N.O^-_{2n}(2)         & 2^{n+1}               & Q_{\infty, 2}       \\ \hline
p_+^{1+2n}\Ydown C_{p^m}, \quad (m>1)       & \pm N.\Sp_{2n}(p)     & p^{m+n-1}(p-1)        & \Q(\zeta_p)           \\ \hline
2_+^{1+2n}\Ydown D_{2^m}, \quad (m>3)       & N.\Sp_{2n}(2)         & 2^{n+m-2}             & \Q(\theta_{2^{m-1}})  \\ \hline
2_+^{1+2n}\Ydown Q_{2^m}, \quad (m>3)       & N.\Sp_{2n}(2)         & 2^{n+m-1}             & Q_{\theta_{2^{m-1}}, \infty}  \\ \hline
2_+^{1+2n}\Ydown QD_{2^m}, \quad (m>3)      & N.\Sp_{2n}(2)         & 2^{n+m-2}             & \Q(\zeta_{2^{m-1}} - \zeta_{2^{m-1}}^{-1})\\ \hline
\end{array}
$$
where $\theta_k = \zeta_k + \zeta_k^{-1}$ generates the maximal totally real subfield of the cyclotomic field $\Q(\zeta_k)$ and $Q_{\theta_{2^{m-1}}, \infty} \isom Q_{\infty,2} \tensordown{\Q} \Q(\theta_{2^{m-1}})$ denotes the quaternion algebra with center $\Q(\theta_{2^{m-1}})$ ramified only at the infinite places.

Note however that $\B(N)$ is only correct under the assumption that $N\unlhd G$.
\end{prop}

The following example shows how the generalized Bravais group can be used to eliminate possible candidates for normal subgroups.
\begin{example}
Let $q \equiv \pm 3 \mod 8$ be a prime power. Suppose $N < \GL_{q}(\Q)$ is isomorphic to $\LGruppe_2(q)$ such that the natural character of $N$ is the Steinberg character. Then $\B(N)$ is conjugate to the automorphism group $\Aut(A_{q}) \isom S_{q+1}$ of the root lattice $A_q$. In particular, there exists no \spimf\ $G < \GL_{kq}(\Q)$ that contains a normal subgroup isomorphic to $\LGruppe_2(q)$ such that its natural character is $k$-times the Steinberg character.
\end{example}
\begin{proof}
The natural representation of $N$ is the $q$-dimensional summand of the permutation module of $\LGruppe_2(q)$ on the projectice space $\PS(\F_q)$. Hence one can assume that $N$ fixes the root lattice $L:= A_q$.

Since $q\ne 7,8$ the group $\Aut(A_q)$ is a maximal finite subgroup of $\GL_q(\Q)$ (see \cite{Burnside}). So it suffices to check that $N$ and $\Aut(A_q)$ have the same invariant lattices. Let $\ell$ be a prime divisor of $\abs{N} = \frac{1}{2}(q-1)q(q+1)$.
If $\ell$ divides $q(q-1)$ then the decomposition matrix of $\LGruppe_2(q)$ \cite{PSL} shows that $L/pL$ is an irreducible $\F_\ell \LGruppe_2(q)$-module.

Now $S_{q+1}$ and $\LGruppe_2(q)$ both act 2-transitively on $\PS(\F_q)$. If $\ell\ne 2$ divides $q+1$ then the corresponding $\F_\ell$-modular representations have $2$ composition factors (again from \cite{PSL}). Hence it follows from \cite[Theorem 5.1]{NTA} that $N$ and $\Aut(A_q)$ have the same invariant lattices. Thus $\B(N) = \Aut(A_q)$ as claimed. The second statement follows immediately from the definition of generalized Bravais groups and Lemma \ref{lemmabravais}.
\end{proof}

\section{Some infinite families}\label{secfam}

\subsection{Some subgroups of $\Sp_{p-1}(\Q)$}

Let $p\ge 5$ be a prime and write $p-1 = 2^a \cdot o$ with $o$ odd. In the spirit of \cite[Chapter V]{FRMG} this section describes all \simf\ supergroups $G$ of $C_p$ in dimension $p-1$ where $C_p$ denotes the (up to conjugacy) unique cyclic matrix group of order $p$ in $\GL_{p-1}(\Q)$.

Clearly, one possibility is that $G$ contains a normal subgroup conjugate to $C_p$. Since the commuting algebra of $C_p$ is isomorphic to $\Q(\zeta_p)$ it follows that $G/\pm C_p \le C_{2^a} \times C_o$. By Galois theory, $G$ is symplectic if and only if $G/\pm C_p \le C_{o}$ and therefore $G \isom \pm C_p \DP C_o$ by maximality. (The group $\pm C_p \DP C_o$ has only one irreducible rational representation of degree $p-1$ and therefore it can be identified with this representation.)

Another class of candidates are extensions of $\LGruppe_2(p)$. The smallest faithful irreducible complex representations of $\LGruppe_2(p)$ are of degree $\frac{p-1}{2}$ and algebraically conjugate. The corresponding character field is $\Q(\sqrt{\pm p})$ with the $-$~sign if and only if $p \equiv -1\mod 4$ (see \cite{SchurPSL}).

If $p \equiv -1 \mod 4$ then $\LGruppe_2(p)$ contains a subgroup $U$ isomorphic to $C_p \DP C_{\frac{p-1}{2}}$. The restriction of the natural representation of $\LGruppe_2(p)$ on $U$ is irreducible and has the same character field $\Q(\sqrt{-p})$ (\cite{SchurPSL}). By \cite[Satz 1.2.1]{Lorenz}, the Schur index of $\LGruppe_2(p)$ is equal to the Schur index of $U$ which is $1$. Thus $C_2 \times \LGruppe_2(p)$ has a unique $p-1$~dimensional rationally irreducible representation (denoted by ${}_{\sqrt{-p}}[\pm \LGruppe_2(p)]_{\frac{p-1}{2}}$ in the sequel) with commuting algebra $\Q(\sqrt{-p})$.

The next result shows that there are no further possibilities. More precisely:

\begin{theorem}\label{pminus1}
Let $p\ge 5$ be prime and $G < \Sp_{p-1}(\Q)$ such that $p$ divides the order $\abs{G}$.  Write $p-1=2^a\cdot o$ with $o$ odd.\\
Then $G$ is \simf\ if and only if $G$ is conjugate to
$
\begin{cases}
\pm C_p \DP C_{o} & \text{ if } p \equiv +1\mod 4\\
{}_{\sqrt{-p}}[\pm \LGruppe_2(p)]_{\frac{p-1}{2}} & \text{ if } p \equiv -1 \mod 4
\end{cases}
$.
\end{theorem}

To state the proof, an upper bound on $\abs{G}$ is needed.
\begin{lemma}[Minkowski's bound, \cite{Minkowski}]\label{Minkowski}
The least common multiple of the orders of all finite subgroups of $\GL_n(\Q)$ is given by
$$\prod_{p} p^{ \left\lfloor \frac{n}{(p-1)}\right\rfloor + \left\lfloor \frac{n}{p(p-1)}\right\rfloor + \left\lfloor \frac{n}{p^2 (p-1)}\right\rfloor + \dots}$$
where the product is taken over all primes $p \le n+1$.
\end{lemma}

\begin{proof}[Proof of Theorem \ref{pminus1}]
The cases $p=5$ and $7$ can be checked explicitly using the 2-parameter argument (see Corollary \ref{mparam}). Let $p \ge 11$ be prime.
Let $\pm I_{p-1} < G < \Sp_{p-1}(\Q)$ such that $p$ divides $\abs{G}$. Further let $P \in \Syl_p(G)$. Then by Minkowski's bound, $\abs{P} = p$. 
So the commuting algebra of $P$ is isomorphic to $\Q(\zeta_p)$. Thus $\pm I_{p-1} \le Z(G) \le C_G(P) = \pm P$. Since $G$ must have a faithful irreducible complex character of degree $\le \frac{p-1}{2}$ it follows from a theorem of H. Blau \cite[VIII Theorem 7.2]{Feit} that either $P\unlhd G$ or $G/Z(G) \isom \LGruppe_2(p)$. In the first case, $G$ is conjugate to a subgroup of $\pm C_{p} \DP C_o$. In the second case, $Z(G) = \pm P$ would imply that $G=\pm P$ since $\pm P$ is self-centralizing in $G$. So $Z(G) = \pm I_{p-1}$. Since $\LGruppe_2(p)$ is perfect, $G$ is either isomorphic to $\pm \LGruppe_2(p)$ or $\SL_2(p)$. By the generic character table of $\LGruppe_2(p)$ (see \cite{SchurPSL}) this implies $p \equiv -1 \mod 4$. But then the real Schur indices of the $\frac{p-1}{2}$ dimensional complex characters of $\SL_2(p)$ are $2$. Thus $G$ is conjugate to ${}_{\sqrt{-p}}[\pm \LGruppe_2(p)]_{\frac{p-1}{2}}$. Since ${}_{\sqrt{-p}}[\pm \LGruppe_2(p)]_{\frac{p-1}{2}}$ contains a subgroup conjugate to $\pm C_p \DP C_o$ the result follows.
\end{proof}

\subsection{Some subgroups of $\Sp_{p+1}(\Q)$}

Let $p \ge 5$ be a prime. If $p \equiv -1\mod 4$ then $G:= \SL_2(p)$ has only two algebraically conjugate complex representations of degree $\frac{p+1}{2}$ as the generic character table \cite{SchurPSL} shows. Let $\chi$ denote one of the corresponding characters and let $P \in \Syl_p(G)$. An explicit calculation shows $(1_P^G, \chi)_G = (1_P, \chi\vert_P)_P =1$. Thus (by \cite[Corollary 10.2(c)]{Isaacs}) $\chi$ is realizable over its character field, which is $\Q(\sqrt{-p})$. So $\chi$ gives rise to a subgroup of $\Sp_{p+1}(\Q)$ denoted by ${}_{\sqrt{-p}}[\SL_2(p)]_\frac{p+1}{2}$.

\begin{theorem}\label{pplus1}
Let $p\ge 11$ be prime and $G < \Sp_{p+1}(\Q)$ such that $p$ divides $\abs{G}$.
Then $G$ is \simf\ if and only if $p \equiv -1\mod 4$ and $G$ is conjugate to ${}_{\sqrt{-p}}[\SL_2(p)]_\frac{p+1}{2}$.
\end{theorem}
\begin{proof}
Let $P \in \Syl_p(G)$. Then again, by Minkowski's bound, $\abs{P}=p$. Thus the natural representation of $P$ splits into twice the trivial one and the $\Q$-irreducible of degree $p-1$. Since there exists an embedding $\delta \colon G \to \GL_m(K)$ for some totally complex number field $K$ with $m\cdot [K:\Q] = p+1$ there is only the possibility $K=\Q(\sqrt{-p})$ and $m=\frac{p+1}{2}$. So the commuting algebra of $\delta(P)$ in $K^{m\times m}$ is isomorphic to $\Q(\zeta_p) \times \Q(\sqrt{-p})$. In particular, $Z(G)$ is isomorphic to a subgroup of $C_{2p} \times C_2$. By Lemma \ref{mainlemma} it follows that $P \ntrianglelefteq G$ and therefore $Z(G) \le C_2 \times C_2$. Since $G$ cannot be imprimitive by Minkowski's bound, $Z(G)$ is not isomorphic to $C_2\times C_2$. This shows $Z(G) = \pm I_{p+1}$. From Blau's theorem \cite[VIII Theorem 7.2]{Feit} it follows that $G$ must be isomorphic to $\pm \LGruppe_2(p)$ or $\SL_2(p)$. Since $K \isom \Q(\sqrt{-p})$ it follows from the character table of $\SL_2(p)$ \cite{SchurPSL} that $p \equiv -1\mod 4$ and that $\pm \LGruppe_2(p)$ has no faithful complex character of degree $\frac{p+1}{2}$. So $G$ must be conjugate to ${}_{\sqrt{-p}}[\SL_2(p)]_\frac{p+1}{2}$ and the result follows.
\end{proof}

The classification of all conjugacy classes of \simf\ subgroups of $\Sp_k(\Q)$ for $k \in \{6, 8\}$ show that the above result also holds for $p=5$. But  the unique \simf\ subgroup of $\Sp_8(\Q)$ whose order is divisible by $7$ is ${}_{\sqrt{-7}}[2.\Alt_7]_4$ (which contains a subgroup conjugate to ${}_{\sqrt{-7}}[\SL_2(7)]_4$).

\subsection{The group $QD_{2^n}$}

Let $n\ge 4$. The group $QD_{2^n} = \left\langle x, y \mid x^{2^{n-1}}, y^2, x^y = x^{2^{n-2}-1}\right\rangle $ has one rationally irreducible representation of degree $2^{n-2}$. This representation has $\Q(\zeta_{2^{n-1}} - \zeta_{2^{n-1}}^{-1})$ as commuting algebra and is denoted by ${}_{\zeta_{2^{n-1}} - \zeta_{2^{n-1}}^{-1}} [QD_{2^n}]_2$.

\begin{prop}\label{QuasiDihedral}
If $n\ge 5$ then ${}_{\zeta_{2^{n-1}} - \zeta_{2^{n-1}}^{-1}} [QD_{2^n}]_2$ is a \simf\ subgroup of $\Sp_{2^{n-2}}(\Q)$.
\end{prop}
\begin{proof}
The commuting algebra $K$ of $H:= {}_{\zeta_{2^{n-1}} - \zeta_{2^{n-1}}^{-1}} [QD_{2^n}]_2$ is isomorphic to $\Q(\zeta_{2^{n-1}} - \zeta_{2^{n-1}}^{-1})$. This is the fixed field of the automorphism of $\Q(\zeta_{2^{n-1}})$ induced by $\zeta_{2^{n-1}}\mapsto -\zeta_{2^{n-1}}^{-1} = \zeta_{2^{n-1}}^{2^{n-2}-1}$. The subfields of $K$ are linearly ordered by Galois theory and the maximal subfield is the fixed-field of complex conjugation i.e. totally real. Thus each \simf\ supergroup $G$ of $H$ embeds into $\GL_2(K)$. Therefore $\bar{G}:= G / Z(G) = G/\left\langle \pm I_{2^{n-2}}\right\rangle $ embeds into $\PGL_2(K)$. Since $n\ge 5$ is assumed one finds that $\bar{G}$ is a dihedral group (of order $2^{n-1}$) according to Blichfeldt's classification \cite{Blichfeldt}. This shows $G=H$ as claimed.
\end{proof}

\subsection{The group $2_+^{1+2n}$}

Let $T_n < \GL_{2^n}(\Q) $ be the $n$-fold tensor product of $\left\langle \left( \begin{smallmatrix} 0&1\\1&0 \end{smallmatrix}\right),\; \left( \begin{smallmatrix} 1&0\\0&-1 \end{smallmatrix}\right) \right\rangle \isom D_8$. So $T_n$ is isomorphic to the extraspecial group $2_+^{1+2n}$. 

This section, which is heavily based on Section 5 of \cite{cliff}, describes the construction of $\B(T_n)$ and defines a maximal finite subgroup of $\GL_{2^n}(\Q(\sqrt{-2}))$ which gives rise to a \spimf\ subgroup of $\Sp_{2^{n+1}}(\Q)$.

Let $(b_0, \dots, b_{2^n-1})$ be the standard basis of $\Q^{1 \times 2^n}$. If one identifies $v\in \F_2^n$ with $j= \sum_i v_i 2^{i-1}$ then the basis vectors $b_j$ can be indexed by elements of $\F_2^n$. For an affine subspace $U$ of $\F_2^n$ let $\chi_U = \sum_{u \in U} b_u$. Then $L_n$ and $L'_n$ are the $\Z$-lattices in $\Q^{1 \times 2^n}$ spanned by
$$ \left \lbrace 2^{\lfloor(n-\dim(U)+\delta)/2\rfloor} \chi_U \mid  U \text{ an affine subspace of } \F_2^n \right\rbrace $$
where $\delta=0$ for $L_n$ and $\delta=1$ for $L'_n$.

In \cite[Theorem 3.2]{Wall} it is shown that $H_n:= \Aut(L_n, I_{2^n}) \cap \Aut(L'_n, I_{2^n})$ is isomorphic to $2_+^{1+2n}.O^+_{2n}(2)$ and further $O_2(H_n)$ is conjugate to $T_n$.

It is shown in \cite{Winter} that $\Out(2_+^{1+2n}) \isom O^+_{2n}(2) \DP 2 \isom GO^+_{2n}(2)$ is the full orthogonal group of a quadratic form of Witt defect $0$. Conjugation by $h_n:= \left(\begin{smallmatrix} 1&1\\1&-1 \end{smallmatrix}\right) \tensor I_{2^{n-1}}$ induces an outer automorphism on $T_n$ which is not realized by $H_n$. Since $\End({T}_n) \isom \Q$ and $h_n^2 = 2 I_{2^n}$, there exists no element in $\GL_{2^n}(\Q)$ of finite order which induces the same automorphism on $T_n$. Thus $H_n$ is conjugate to $\B(T_n)$ (once it is shown that $T_n \unlhd \B(T_n)$). Moreover $\frac{1}{\sqrt{-2}} h_n$ normalizes $T_n$ and therefore $\B(T_n)$. 
A comparison of the orders of $H_n$ and $\Out(T_n)$ shows that $\calH_n:= \left\langle H_n,\; \frac{1}{\sqrt{-2}} h_n\right\rangle$ is the unique extension of $H_n$ by $C_2$ in $\GL_{2^n}(\Q(\sqrt{-2}))$. The group $\calH_n$ gives rise to a finite symplectic matrix group in $\Sp_{2^{n+1}}(\Q)$ which will be denoted by ${}_{\sqrt{-2}} [2_+^{1+2n}.(O^+_{2n}(2) \DP 2)]_{2^n}$ in the sequel.

Finally, by extending scalars $M_n := \sqrt{-2} L'_n + L_n$ is a $\Z[\sqrt{-2}]$-lattice generated by 
$$\left\lbrace \sqrt{-2}^{n-\dim(U)} \chi_U \mid U \text{ an affine subspace of } \F_2^n \right\rbrace\:.$$

\begin{lemma}
The lattice $M_n$ is the $n$-fold tensor product $M_1 \tensordown{\Z[\sqrt{-2}]} \dots  \tensordown{\Z[\sqrt{-2}]} M_1$.
\end{lemma}
\begin{proof}
Let $V_{n-1} = \left\langle e_1, \dots, e_{n-1} \right\rangle$ and $V_1 = \left\langle e_n \right\rangle$. One checks that $b_x \tensor b_y = b_{x+y}$ for all $x\in V_{n-1}, y \in V_1$.

Let $U$ be a $d$-dimensional affine subspace of $V_{n-1}$. For $y$ in $V_1$ it follows that $\sqrt{-2}^{n-1-d}\chi_U \tensor \sqrt{-2} b_y = \sqrt{-2}^{n-d} \chi_{y+U} \in M_n$. Similarly, since $U + V_1$ has dimension $d+1$, it follows that $\sqrt{-2}^{n-1-d} \chi_U \tensor \chi_{V_1} = \sqrt{-2}^{n-(d+1)} \chi_{U+V_1} \in M_n$. Thus $M_{n-1} \tensor M_1 \subseteq M_n$.

Conversely, suppose $U$ is a $d$-dimensional affine subspace of $\F_2^n$. Write $U = x+y + U_0$ where $U_0$ is a subspace of $\F_2^n$ and $x \in V_{n-1}$, $y \in V_1$. 

If $U_0 \le V_{n-1}$ then $\sqrt{-2}^{n-d} \chi_U = \sqrt{-2}^{n-1-d} \chi_{x+U} \tensor \sqrt{-2} b_y \in M_{n-1} \tensor M_1$. Otherwise $U_{n-1} := U_0 \cap V_{n-1}$ is a $(d-1)$-dimensional subspace and $U_0 = U_{n-1} \cup (z+e_n + U_{n-1})$ for some $z \in V_{n-1}$.

If $z \in U_{n-1}$, then $\sqrt{-2}^{n-d} \chi_U = \sqrt{-2}^{n-1 - (d-1)} \chi_{x+U_{n-1}} \tensor \chi_{V_1}$ otherwise
\begin{align*}
\sqrt{-2}^{n-d} \chi_U &= 
\sqrt{-2}^{n-1-d} \chi_{x+U_{n-1}+\left\langle z \right\rangle} \tensor \sqrt{-2}b_y\\
&+ \sqrt{-2}^{n-1-(d-1)} \chi_{x+z+U_{n-1}}\tensor \chi_{V_1} \\
&- \sqrt{-2}(\sqrt{-2}^{n-1-(d-1)} \chi_{x+z+U_{n-1}}\tensor \sqrt{-2}b_y )
\end{align*}
This shows $M_n \subseteq M_{n-1} \tensor M_1$.
\end{proof}

\begin{lemma} The group $\calH_n$ is conjugate to
the Hermitian automorphism group $\Aut_{\Q(\sqrt{-2})}(M_n) := \Stab_{U_{2^n}(\Q(\sqrt{-2}))}(M_n)$ in $\GL_{2^n}(\Q(\sqrt{-2}))$.
\end{lemma}
\begin{proof}
For $n=3$ the result can be checked explicitly. So let $n \ne 3$ denote by $(v_1,\dots,v_{2^n})$ a $\Z$-basis of $L'_n$ such that $(2v_1,\dots,2v_{2^{n-1}}, v_{2^{n-1}+1},\dots,v_{2^n})$ is a $\Z$-basis of $L_n$. Then $(\sqrt{-2}v_1,\dots,\sqrt{-2}v_{2^n},2v_1,\dots,2v_{2^{n-1}}, v_{2^{n-1}+1},\dots,v_{2^n})$ is a $\Z$-basis of $M_n = \sqrt{-2} L'_n \oplus L_n$. In particular, the $\Z$-lattices $L_n$ and $\sqrt{-2}L'_n$ are perpendicular with respect to the scalar product $(x,y) \mapsto \frac{1}{2} \Tr_{\Q(\sqrt{-2})/\Q}(x\transpose{\bar{y}})$.
Thus the group $\Aut_{\Q(\sqrt{-2})}(M_n)$ is the subgroup of $\Aut(\sqrt{-2} L'_n \perp L_n, I_{2^{n+1}})$ which commutes with $\sqrt{-2}$. Since $n \ne 3$ is assumed, the automorphism groups of $L_n$ and $L'_n$ equal $H_n$ \cite[Theorem 3.2]{Wall}. Hence with respect to appropriate $\Z$-bases,  $\Aut_{\Q(\sqrt{-2})}(M_n)$ contains a subgroup $G_n$ of index at most two where
$$G_n = \left\lbrace \begin{pmatrix} g_1 & 0 \\ 0 & g_2\end{pmatrix} \mid g_1, g_2 \in H_n \right\rbrace \cap \Aut_{\Q(\sqrt{-2})}(M_n)\:.$$
Now $\sqrt{-2}$ interchanges $\sqrt{-2}L'_n$ and $L_n$, i.e. it operates as a block matrix $\left(\begin{smallmatrix} 0 & w \\ -2w^{-1} & 0\end{smallmatrix}\right)$ for some $w \in \GL_{2^n}(\Q)$. So $\left(\begin{smallmatrix} g_1 & 0 \\ 0 & g_2\end{smallmatrix}\right) \in G_n$ if and only if $g_2=g_1^w$. Therefore the subgroup $G_n< \Aut_{\Q(\sqrt{-2})}(M_n)$ is conjugate to $H_n$ in $\GL_{2^n}(\Q(\sqrt{-2}))$.
The result follows since $\frac{1}{\sqrt{-2}} h_n$ acts on $M_1 \tensor M_{n-1} = M_n$.
%On the other hand, $\frac{1}{\sqrt{-2}} h_n$ acts on $M_1 \tensor M_{n-1} = M_n$. So $\Aut_{\Q(\sqrt{-2})}(M_n)$ is isomorphic to $\calH_n$. Both groups have conjugate Fitting subgroups and therefore their generalized Bravais groups must be conjugate as well. But then the whole groups must be conjugate, since $\calH_n$ is the unique extension of $H_n$ in $\GL_{2^n}(\Q(\sqrt{-2}))$ by the extra automorphism induced by $h_n$.
\end{proof}

\begin{theorem}\label{e2plus}
If $n \ge 2$, then ${}_{\sqrt{-2}}[2_+^{1+2n}.(O^+_{2n}(2) \DP 2)]_{2^n}$ is the (up to conjugacy) unique \spimf\ subgroup of $\Sp_{2^{n+1}}(\Q)$ with Fitting group $2_+^{1+2n}$. In particular, $\B(T_n)$ is conjugate to $H_n$.
\end{theorem}
\begin{proof}
Since $\Z[\sqrt{-2}]$ is a PID, $M_n$ has a $\Z[\sqrt{-2}]$-basis. With respect to this basis, its automorphism group $G_n$ is a finite subgroup of $\GL_{2^n}(\Z[\sqrt{-2}])$. By explicit calculations in \MAGMA\ \cite{Magma} one checks that for $n\in\{2,3\}$ the $\Z$-span $\left\langle G_n \right\rangle_\Z$ equals $\Z[\sqrt{-2}]^{2^n \times 2^n}$. For $n \ge 4$ it follows from $M_n = M_{n-2} \tensordown{\Z[\sqrt{-2}]}M_2$ that $G_2 \tensor G_{n-2} \subset G_{n}$ (using appropriate bases). By induction, one gets
$$\Z[\sqrt{-2}]^{2^n \times 2^n} = \left\langle G_{n-2} \right\rangle_\Z \tensordown{\Z[\sqrt{-2}]} \left\langle G_{2} \right\rangle_\Z \subseteq \left\langle G_{n} \right\rangle_\Z \subseteq \Z[\sqrt{-2}]^{2^n \times 2^n}\:.$$
In particular, each $G_n$-invariant $\Z[\sqrt{-2}]$-lattice is a multiple of $M_n$, since $\Z[\sqrt{-2}]$ has class number 1. Thus $G_n$ is a maximal finite subgroup of $\GL_{2^n}(\Q(\sqrt{-2}))$. But any finite symplectic supergroup of $G:= {}_{\sqrt{-2}}[2_+^{1+2n}.(O^+_{2n}(2) \DP 2)]_{2^n}$ comes from a finite supergroup of $\calH_n\isom G_n < \GL_{2^n}(\Q(\sqrt{-2}))$. Thus $G$ is \simf{}.

As explained in the beginning of this section, to prove that $\B(T_n)$ is conjugate to $H_n$, it suffices to show that $T_n \unlhd \B(T_n)$. But this is clear from Lemma \ref{lemmabravais} since (a rational constituent of) $F(G)$ is conjugate to $T_n$.

Suppose now $S < \Sp_{2^{n+1}}(\Q)$ is \spimf\ such that $F(S) \isom 2^{1+2n}$. Then by Lemma \ref{mainlemma}, $F(S)$ must be conjugate to $U$, so one can assume $U < S$. But then $\B(U) \unlhd S$.  Now $\abs{\Out(2^{1+2n})} = 2 [\B(U): U]$ shows that $G$ and $S$ contain $\B(U)$ with index $2$. Since $\End(U) \isom \Q^{2\times 2}$ it follows from Theorem \ref{QA} that $G$ is the (up to conjugacy) unique symplectic extension of $\B(U)$ by $C_2$.
\end{proof}

\subsection{The group $p_+^{1+2n}$}

In this section, let $p$ be an odd prime. In this section, a family of irreducible symplectic matrix groups in dimension $p^n(p-1)$ is described which will be maximal finite in the case that $p$ is a Fermat prime, i.e. $p-1$ is a power of two.

Let $T_n^{(p)} \isom p_+^{1+2n}$ be the $n$-fold tensor product of $T_1^{(p)}$ where $T_1^{(p)}$ is the subgroup of $\GL_p(\Q(\zeta_p))$ generated by the diagonal matrix $\Diag(1, \zeta_p, \dots, \zeta_p^{p-1})$ and the permutation matrix corresponding to the $p$-cycle $(1,\dots,p)$. Further let $H_n^{(p)} = N_{U_{p^n}(\Q(\zeta_p))}(T_n^{(p)})$. By \cite{Winter} $H_n^{(p)}$ is isomorphic to a subgroup of $C_2 \times p_+^{1+2n}.\Sp_{2n}(p)$, since the group of outer automorphisms which act trivially on the center of $T_n^{(p)}$ is isomorphic to $\Sp_{2n}(p)$. 

In \cite[Section 4]{Wall} Wall constructs a $(p-1)p^n$ dimensional rational lattice on which $C_2 \times p_+^{1+2n}.\Sp_{2n}(p)$ acts. Thus $H_n^{(p)} = \B(T_n^{(p)}) \isom C_2 \times p_+^{1+2n}.\Sp_{2n}(p)$ (if $T_n^{(p)} \unlhd \B(T_n^{(p)})$ it assumed). Now $H_n^{(p)}$ gives rise to a finite subgroup of $\Sp_{p^n(p-1)}(\Q)$ which will be denoted by ${}_{\zeta_p}[\pm p_+^{1+2n}.\Sp_{2n}(p)]_{p^n}$.

\begin{theorem}\label{Fermat}
If $p$ is a Fermat prime, then ${}_{\zeta_{p}}[\pm p_+^{1+2n}.\Sp_{2n}(p)]_{p^n}$ is a \simf\ subgroup of $\Sp_{p^n(p-1)}(\Q)$.
\end{theorem}
\begin{proof}
The commuting algebra $C$ of ${}_{\zeta_{p}}[\pm p_+^{1+2n}.\Sp_{2n}(p)]_{p^n}$ is isomorphic to $\Q(\zeta_p)$. Since $p$ is a Fermat prime, $C$ has only one maximal subfield, which must be totally real. Thus any symplectic supergroup must come from a finite subgroup of $\GL_{p^n}(\Q(\zeta_p))$. But $H_n^{(p)}$ is maximal finite in $\GL_{p^n}(\Q(\zeta_p))$ according to \cite[Theorem 7.3]{cliff}.
\end{proof}

\section{Dimension 2p}\label{sec2p}

Let $p\ge 5$ be prime. The structure of \spimf\ subgroups of $\Sp_{2p}(\Q)$ is rather restricted.

\begin{lemma}
If $G < \Sp_{2p}(\Q)$ is \spimf\ then $F(G)$ is cyclic of order $2, 4$ or $6$.
\end{lemma}
\begin{proof}
Clearly, if $q \ge 5$ is a prime such that $O_q(G) \ne 1$ then $q=2p+1$ by Lemma \ref{mainlemma}. In particular, $q\equiv -1 \mod 4$ and $q \ge 11$. This contradicts Theorem \ref{pminus1}. So $F(G) = O_2(G)O_3(G)$. If $G$ would be cyclic, then $G=F(G)= O_2(G)O_3(G)$ is reducible. So $G$ embeds into $\GL_p(K)$ for some imaginary quadratic number field $K$. Hence the result follows from Hall's classification (see Theorem \ref{Hall}).
\end{proof}

Hence the generalized Fitting subgroup $F^*(G)$ of a \spimf\ group $G < \Sp_{2p}(\Q)$ cannot be soluble. Thus $F^*(G)$ is either reducible with $\Q^{2\times 2}$ as commuting algebra or irreducible and its commuting algebra is an imaginary quadratic number field. Thus $G$ can easily be recovered from $F^*(G)$ by the methods in the next section.

\section{Some methods}\label{secmet}

\subsection{Normal subgroups of index $2^k$}\label{subsecind2}

If $G < \Sp_{2n}(\Q)$ is maximal finite, then $G/\B(F^*(G))$ is very often a $2$-group. Thus two criteria are given that eliminate some candidates for $F^*(G)$ in these cases.

The first lemma is an analogon to \cite[Corollary III.4]{FRMG}.

\begin{lemma}
Let $N$ be a subgroup of a \spimf\ group $G < \Sp_{2n}(\Q)$ of index $2$. If $N$ is rationally reducible then $\End(N) \isom L^{2\times 2}$ for some totally real number field $L$.
\end{lemma}
\begin{proof}
Let $g \in G - N$.
By Cifford theory, the natural representation $\triangle$ of $G$ splits into $\triangle_1 + \triangle_1^g$ for some irreducible $\triangle_1 \colon N \to \GL_n(\Q)$. Thus $G$ is conjugate to $\left \langle I_2 \tensor \triangle_1(N), \left( \begin{smallmatrix} 0 & I_n \\ \triangle_1(g^2) & 0 \end{smallmatrix} \right) \right\rangle \le \triangle_1(N)\wr C_2$. Since $G$ is symplectic primitive and $\triangle_1(N)$ is rationally irreducible this implies $\dim_\Q(\Forms_{skew}(\triangle_1(N))) = 0$.

Let $\delta\colon N \to \GL_n(\R), h \mapsto \triangle_1(h)$ decompose into $\R$-irreducible representations $\delta_1, \dots, \delta_s$. It follows from the above that $\Forms(\delta(N)) = \Forms_{sym}(\delta(N))$. In particular, this implies that the $\delta_i$ are pairwise nonisomorphic and $\End(\delta_i(N)) \isom \R$.
Hence $L \tensor_\Q \R \isom \oplus_{i=1}^s\R$ and therefore $L$ is a totally real number field.
\end{proof}

The smallest example of such a group $N$ is $N = \left\langle -I_2 \right\rangle \unlhd \left\langle \left( \begin{smallmatrix} 0 & 1 \\ -1 & 0 \end{smallmatrix} \right) \right\rangle < \Sp_{4}(\Q)$ which has $\Q^{2\times 2}$ as commuting algebra.

\begin{example}
The Fitting subgroup of a \spimf\ group $G < \Sp_8(\Q)$ is not isomorphic to $Q_8$.
\end{example}
\begin{proof}
 Suppose otherwise. Then $\End(F(G)) \isom Q_{\infty,2}^{2\times 2}$ where $Q_{\infty,2}$ denotes the quaternion algebra over $\Q$ ramified only at $2$ and $\infty$. Thus $E(G)=1$ by \cite{HissMalle}. Hence $G/F(G) \le \Out(F(G)) \isom S_3$. Since $\B(F(G))\isom \SL_2(3)$ one gets $[G: \B(F(G))] \le 2$. This contradicts the result above since $\End(F(G))  = \End(\B(F(G))) \isom Q_{\infty, 2}^{2\times 2}$. 
\end{proof}

If $n$ is not a power of $2$, the following lemma can be used to rule out several candidates for normal subgroups.

\begin{lemma}
Let $G< \Sp_{2n}(\Q)$ be rationally irreducible and symplectic primitive. Let $N \unlhd G$ such that $[G:N] =2^k$. If $\triangle$ denotes the natural representation of $N$ then $\triangle = 2^\ell \cdot \delta$ for some irreducible representation $\delta$ of $N$ and some $0 \le \ell \le k$.
\end{lemma}
\begin{proof}
The proof given in \cite[Lemma III.4]{ERMG} applies mutatis mutandis.
\end{proof}

\begin{example}
Let $G < \Sp_{12}(\Q)$ be \spimf{}. If $F(G)=O_2(G)$, then $O_2(G)$ is isomorphic to $C_2, C_4$ or $D_8$ and $E(G) \ne 1$.
\end{example}
\begin{proof}
Let $U = \{Q_8,\; C_8,\; D_{16},\; QD_{16},\; D_8 \tensor C_4,\; 2^{1+4}_+\}$. Hall's theorem (see Theorem \ref{Hall}) shows that $O_2(G) \in \{C_2, C_4, D_8\} \cup U$. 
Suppose first that $E(G) = 1$. Then $G/\B(F(G)) \le \Out(F(G))$ is a $2$-group. By the previous result this implies that the natural representation of $\B(F(G))$ must be a multiple of a faithful rationally irreducible representation of degree $3, 6$ or $12$. But $F(G) = O_2(G)$ and thus $\B(F(G))$ does not admit such representations. This implies $E(G) \ne 1$.
Suppose now $O_2(G) \in U$. Then $\End(O_2(G)) \isom Q^{3 \times 3}$ where $Q$ is isomorphic to $\Q$, $\Q(i)$, $\Q(\sqrt{\pm 2})$ or $Q_{\infty, 2}$. But no quasisimple group embeds in to $Q^{3 \times 3}$ (see \cite{HissMalle}). Thus $E(G)=1$ gives the desired contradiction.
\end{proof}

\subsection{m-parameter argument}\label{subsecmpar}

Suppose $U < \GL_{n}(\Q)$ s finite such that $E:= \End(U)$ is a number field of absolute degree $m$ and let $F_0 \in \tpForms(U)$. In particular, this applies to irreducible cyclic matrix groups. If $E^+$ denotes the maximal totally real subfield of $E$ then $\tpForms(U) = E^+ \cdot F_0$. The thesis \cite{ERMG} gives an algorithm which uses this parametrization and ideal arithmetic in $E^+$ to show that every finite supergroup $G\ge U$ fixes (up to conjugacy) one of the forms $cF$ where $c$ runs through some finite set.

Two more definitions are needed to  state the result.

\begin{definition}
For a totally real number field $K$, denote by $\Pi(K)$ a finite set of prime numbers such that
\begin{itemize}
\item For each $x \in K$ there exists some $y \in \Z_K$ whose norm is supported at $\Pi(K)$ such that $xy$ is totally positive. 
\item In each ideal class of $\Z_K$ there exists some integral ideal that contains a natural number whose prime factors come from $\Pi(K)$.
\end{itemize}
Furthermore, if $\ell \in \Z$ set
$$\tilde{\Pi}(\ell, K) = \{p \mid\mbox{$p$ is a prime divisor of $\ell$ or $p \in \Pi(k)$ for some subfield $k$ of $K$}\}\:.$$
\end{definition}

\begin{definition}
Let $L$ be a lattice in $\Q^{1\times n}$ of full rank and let $F \in \Q^{n\times n}$ be symmetric positive definite.
\begin{enumerate}
\item $L^{\#, F} = \{x \in \Q^{n \times 1} \mid xF\transpose{y} \in \Z \mbox{ for all } y \in L\}$
denotes the dual of $L$ with respect to $F$. 
\item The pair $(L, F)$ is said to be integral if $L \subset L^{\#, F}$.
\item The integral pair $(L, F)$ is said to be normalized if the abelian group $L^{\#, F} / L$ has squarefree exponent and its rank is at most $n/2$.
\end{enumerate}
\end{definition}

\begin{theorem}[\protect{\cite[Korollar III.3]{ERMG}}]\label{mparam}
Suppose $U < \GL_{n}(\Q)$ is finite such that $E:= \End(U)$ is a field and denote by $E^+$ its maximal totally real subfield. Let $U < G <\GL_n(\Q)$ be finite and fix some $L \in \ZZ(G)$. Then there exists some $F \in \tpForms(U)$ such that $(L, F)$ is integral and the prime divisors of $\det(L, F)$ are contained in $\tilde{\Pi}(\abs{G}, E^+)$.
\end{theorem}

To find all (conjugacy classes of) \simf\ matrix groups $G$ that contain a conjugate copy of $U$, one has to consider the groups $\Aut_K(L, F)$ for some minimal totally complex subfield $K$ of $E$ and some integral pair $(L, F)$ where the prime divisors of $\det(L, F)$ are bounded by $\tilde{\Pi}((n+1)!, E^+)$ (see Minkowski's bound).

The following ideas reduce the number of pairs $(L,F)$ that one has to consider to a finite number (see \cite[Chapter VII]{ERMG} and \cite[Remark 2.2.13]{MK} for details).
\begin{itemize}
\item
If $G$ fixes some integral $(L, F)$ which is not normalized, then exists some prime $p$ dividing $\det(L, F)$ such that  $(L \cap p^{-1}L^{\#, F}, F)$ is integral. Iterating this process, one finds some normalized pair $(L', F) \in \ZZ(G)\times \tpForms(G)$.
\item 
Suppose $L_1, \dots, L_r$ represent the isomorphism classes of $\left\langle U \right\rangle_\Z$-invariant lattices in $\Q^{n \times n}$.
Then $L'$ is isomorphic to some $L_i$ i.e. there exists some $x \in \GL_n(\Q)$ such that $L' = L_i x$. 
\item
Since $G= \Aut_K(L', F) = \Aut_K(L_ix, F) = x\Aut_K(L_i, x^{-1} F \mtranspose{x})x^{-1}$ one only has to consider the normalized pairs $(L_i, F)$ with $F \in \tpForms(U)$ such that the prime divisors of $\det(L_i, F)$ are contained in $\tilde{\Pi}((n+1)!, E^+)$.
\item
If $y \in E$ such that $L_i y = L_i$ then $\Aut_K(L_iy, F)$ is conjugate to $\Aut_K(L_i, y^{-1}F\mtranspose{y}) = \Aut_K(L_i, \Nr_{E/E^+}(y^{-1})F)$.
\end{itemize}

\subsection{Quaternion algebras as endomorphism rings}\label{subsetquat}

Suppose $N < \Sp_{2n}(\Q)$ is finite such that $Q:= \End(N)$ is a quaternion algebra, i.e. a central simple algebra of degree $4$ over its center $K$.
The next theorem explains how to find (up to conjugacy) all \simf\ supergroups $G$ of $N$ such that $N \unlhd G$. After replacing $N$ by $\B(N)$, one may assume that $N=\B(N)$.

\begin{theorem}\label{QA}
Suppose $x \in \GL_{2n}(\Q)$ commutes with $K$ and acts on $N$ such that $x^2 \in N$.
\begin{enumerate}
\item $\End(\left\langle N, x \right\rangle) \isom K[X] / (X^2-\nr(a))$ for some $a \in Q$ where $\nr\colon Q \to K$ denotes the reduced norm.
\item If $y \in \GL_{2n}(\Q)$ induces the same outer automorphism on $N$ as $x$ then $\End(\left\langle N, y \right\rangle) \isom K[X] / (X^2-u\nr(a))$ for some torsion unit $u \in K^* \cap \nr(Q^*)$.
\item If $\End(\left\langle N, x \right\rangle)$ is a field and $u \in (K^*)^2$ then $\left\langle N, x \right\rangle$ and $\left\langle N, y \right\rangle$ are conjugate.
\end{enumerate}
\end{theorem}
\begin{proof}
\begin{enumerate}
\item For any $z \in Q$ one has $z^x \in \End(N^x) =  \End(N) = Q$. Hence $x$ induces a $K$-automorphism on $Q$. By the Skolem-Noether theorem there exists some $a \in Q^*$ such that $z^a = z^x$ for all $z \in Q$.  Since $x$ does not centralize $Q$ (otherwise $x \in \B(N)= N$ by Lemma \ref{lemmabravais}) it follows that
$$K \subsetneq K[a] \subseteq C_Q(a) = \End(\left\langle N, x \right\rangle)) \subsetneq Q\:. $$
Thus $K[a] = \End(\left\langle N, x \right\rangle)$. Since $x^2 \in N$, it induces the identity on $Q$. Hence $a^2 \in K$ and therefore $K[a] \cong K[X]/(X^2-\nr(a))$.
\item Since $yx^{-1}$ induces an inner automorphism on $N$, it is contained in $N Q^*$. Say $yx^{-1} = g c$ with $g \in N$ and $c \in Q$. Then
$$\End({\left\langle N, y\right\rangle} )= \End({\left\langle N, cx \right\rangle} ) \isom K[X] / (X^2 - \nr(ca))\:.$$
Further $cxcx^{-1} \in (g^{-1}y)^2 x^{-2} \in N$ has finite order. Hence the reduced norm $\nr(c)^2 = \nr(c \cdot (xcx^{-1})) \in K^*$ also has finite order. Thus $u:= \nr(c)$ is a torsion unit as claimed.
\item Let $v\in K$ such that $v^2 = u$. Since $v$ has finite order it is contained in $\B(N) = N$ by Lemma \ref{lemmabravais}. The elements $a^2$ and $(v^{-1}ca)^2$ are both contained in $K$. So $\nr(v^{-1}c)=1$ implies that $a$ and $v^{-1}ca$ have the same minimal polynomial over $K$. It follows from the Skolem-Noether theorem that there exists some $t\in Q^*$ such that $a^t=v^{-1}ca$. Finally
$$x^t = t^{-1} xtx^{-1}x = t^{-1} ata^{-1} x = v^{-1} ca a^{-1} x = (gv)^{-1} y $$
shows $\left\langle N, x \right\rangle^t = \left\langle N, (gv)^{-1} y \right\rangle = \left\langle N, y \right\rangle$.
\end{enumerate}
\end{proof}

In most cases, the structure of $Q$ allows some simplifications:
\begin{enumerate}
\item If $K \isom \Q$ then obviously every $x \in \GL_{2n}(\Q)$ commutes with $K$. 
\item If $Q$ is totally definite then $\End(\left\langle N, x \right\rangle)$ must be a field and there exists no nontrivial torsion unit in $K^* \cap \nr(Q^*)$. I.e. each class of outer automorphisms gives rise to at most one conjugacy class  $\left\langle N, x \right\rangle$.
\item Suppose one wants to find all \spimf\ supergroups $G$ of $N$ satisfying $N = \B(F^*(G))$. If $Q$ is totally definite and $K \isom \Q$ then the proof of \cite[Theorem 4]{PrimStructure} shows that the exponent of $G/N$ is at most $2$. Hence one gets all such groups $G$ (up to conjugacy) with the above result and the $2$-parameter argument.
\end{enumerate}

Here is a simple example.

\begin{example}
Suppose $G < \Sp_4(\Q)$ is \spimf\ and it contains a normal subgroup $U$ isomorphic to $Q_8$. (Note that the character of $U$ is uniquely determined by Lemma \ref{mainlemma}.)

Let $N:= \B(U) \isom \SL_2(3)$. Then $\End(U) = \End(N) \isom Q_{\infty,2}$ is the quaternion algebra over $\Q$ ramified only at $2$ and $\infty$. If there exists some $x \in C_G(N) - N$ then $x$ is contained in some maximal order of $Q$. These orders are all conjugate and have $\SL_2(3)$ as group of torsion units. Thus one can assume that the order of $x$ equals $3$ or $4$. In both cases, the commuting algebra of $H:= \left \langle N ,x \right \rangle$ is an imaginary quadratic number field. If $x$ has order $3$ then $H$ is \simf. In the other case, it follows from the $2$-parameter argument that $H$ is only contained in the \simf\ group $\B(D_8 \tensor C_4) = (D_8 \tensor C_4).S_3$. 
Suppose now $N$ is self-centralizing in $G$. One immediately finds the 4-dimensional faithful representation of $\GL_2(3)$ with commuting algebra $\Q(\sqrt{-2})$ as one possibility for $G$. By the above result, $\GL_2(3)$ is (up to conjugacy) the only possibility since $\Out(N) \isom C_2$ and $\Q^* \cap \nr(Q_{\infty, 2})\subseteq \Q_{>0}$ contains no nontrivial torsion unit.
\end{example}

\section{Tables}\label{sectab}

A database for the the conjugacy classes of \spimf\ subgroups of $\GL_{2n}(\Q)$ up to $n=11$ will be available in future versions of \MAGMA. In the mean time, the database is
available from \url{http://www.math.rwth-aachen.de/~Markus.Kirschmer/symplectic/} .

Loading the file in \MAGMA\ via the {\tt load} command generates an associative array called {\tt SGroups}.
The data for dimension $2n$ is stored in {\tt SGroups[2n]} as a sequence of triples {\tt <FF, o, s>}.
\begin{enumerate}
\item {\tt FF} is a sequence of two forms $F,S$ such that 
\begin{align*}
A &= \Aut_{\Q[SF^{-1}]}(\Z^{1\times 2n}, F)\\
  &= \Aut(\Z^{1\times 2n}, \{F,S\}) := \{g \in \GL_{2n}(\Z)\mid gF\transpose{g}=F \mbox { and } gS\transpose{g}=S\}
\end{align*}
represents the corresponding conjugacy class.
\item The forms $F,S$ satisfy the following properties:
\begin{enumerate}
\item $F\in \tpForms(A)$ and $S \in \Forms_{skew}(A)$.
\item The pair $(\Z^{1\times 2n}, F) \in \ZZ(A)\times \tpForms(A)$ is normalized and has minimal determinant among all normalized $(L, F') \in  \ZZ(A)\times \tpForms(A)$.
\item $SF^{-1}$ is a primitive element of the field $\End(A)$ such that its minimal polynomial $\mu(SF^{-1}, X)$ is one of
\begin{itemize}
\item $X^2+d$ for some squarefree $d \in \N$.
\item $\mu(\zeta_k-\zeta_k^{-1}, X)$ for some even $k \in \Z_{\ge 6}$.
\item $\mu(\zeta_{26} + \zeta_{26}^3 + \zeta_{26}^9, X)$.
\item $\mu(\sqrt{k}\cdot(\zeta_{\ell}-\zeta_{\ell}^{-1}), X)$ where $(k, \ell) \in \{(2, 10),\, (3,10),\, (3,16)\}$.
\item $\mu(\sqrt{-1}+\sqrt{-3}+\sqrt{-5}, X)$.
\end{itemize}
\end{enumerate}
\item {\tt o} is the order of the groups in the conjugacy class of $A$
\item {\tt s} is a string which contains the name of the class of $A$ given in my thesis ($\tensor$ is abbreviated t and so on...).
\end{enumerate}

\begin{example}
Suppose one wants to access the group $\GL_2(\F_3) < \GL_4(\Q)$ in the database. From
\begin{verbatim}
  [x[3] : x in SGroups[4]];
\end{verbatim}
which yields the list:
\begin{verbatim}
  > [ D8tC4.S3, C4tA2, GL23, SL23oC3, C10 ]
\end{verbatim}
one sees that it is the third entry of SGroups[4]. So
\begin{verbatim}
  A:= AutomorphismGroup(SGroups[4,3,1]);
\end{verbatim}
generates a representative {\tt A} of that class.
\end{example}

\begin{remark}
If $G < \GL_{2n}(\Q)$ is \spimf, one can use the information in the database to recover the conjugacy class of $G$ as follows.

If $L_1,\dots,L_s$ represent the isomorphism classes of $G$-invariant lattices, then there exist only finitely many forms $F \in \tpForms(G)$ such that $(L_i, F)$ is normalized and of minimal determinant. After a change of base (i.e. conjugating $G$ with a basis matrix of $L_i$), one may assume that $L_i = \Z^{1\times 2n}$. Then for each $F$ there exist at most $[\End(G):\Q]$ skewsymmetric forms $S \in \Forms(G)$ such that $(F,S)$ satisfies the properties from above. For each such pair, one checks whether there exists some simultaneous isometry between $(F,S)$ and one of the pair of forms $(F', S')$ in the database. I.e. one tests (using the approach described in \cite{Souvignier} which is available in \MAGMA\ via the {\tt IsIsometric} command) whether there exists some $T \in \GL_{2n}(\Z)$ such that $TF\transpose{T} = F'$ and $TS\transpose{T} = S'$.
Such a matrix $T$ exists if and only if $G$ and $A:= \Aut(\Z^{1\times 2n}, \{F',S'\})$ are conjugate. If so, then clearly $\Aut(\Z^{1\times 2n}, \{F,S\}) = A^T$. So after reverting the base change from above, one ends up with a conjugating element between $G$ and $A$.
\end{remark}

\bibliography{sympl}

\end{document}